\documentclass{birkau}

\usepackage{amsmath,amssymb,url}
\usepackage{amssymb}
\usepackage [all]{xy}
\usepackage{mathrsfs}
\usepackage{mathtools}
\usepackage{forest}
\usepackage{stmaryrd}
\usepackage{lmodern}
\usepackage{faktor}

\numberwithin{equation}{section}

\theoremstyle{plain}
\newtheorem{theorem}{Theorem}[section]
\newtheorem{lemma}[theorem]{Lemma}
\newtheorem{proposition}[theorem]{Proposition}

\theoremstyle{definition}
\newtheorem{definition}[theorem]{Definition}
\newtheorem{notation}[theorem]{Notation}
\newtheorem{remark}[theorem]{Remark}
\newtheorem{example}[theorem]{Example}

\DeclareMathOperator{\lcm}{lcm}

\DeclareMathOperator{\Indec}{Ind}
\DeclareMathOperator{\First}{First}

\newcommand{\Ini}{\mathfrak{I}}
\newcommand{\Full}{\mathfrak{F}}

\newcommand{\langleD}{\langle\!\langle}
\newcommand{\rangleD}{\rangle\!\rangle}

\newcommand{\g}{\mathbf{g}}
\newcommand{\I}{\mathcal{I}}
\newcommand{\Nat}{\mathbb{N}}
\newcommand{\M}{\mathbb{M}}
\newcommand{\Tree}{\mathbf{Tree}}
\newcommand{\Lan}{\mathbf{Lan}}
\newcommand{\Fin}{\mathbf{Fin}}
\newcommand{\Seq}{\mathbf{Seq}}
\newcommand{\Per}{\mathbf{Per}}
\newcommand{\E}{\mathbb{E}}
\newcommand{\G}{\mathcal{G}}

\newcommand{\primerexemple}{{\fontsize{2.5}{4}\selectfont
\begin{forest}
 for tree={circle,draw,l sep=1pt, s sep=10pt}
[ [ [][]  ][ [][ [][] ] ]]
\end{forest}
}}

\newcommand{\dos}{{\fontsize{2.5}{4}\selectfont
\begin{forest}
 for tree={circle,draw,l sep=1pt, s sep=10pt}
[ [][] ]
\end{forest}
}}

\newcommand{\trespositiu}{{\fontsize{2.5}{4}\selectfont
\begin{forest}
 for tree={circle,draw,l sep=1pt, s sep=10pt}
[ [][  [] [ ] ] ]
\end{forest}
}}

\newcommand{\producte}{{\fontsize{2.5}{4}\selectfont
\begin{forest}
 for tree={circle,draw,l sep=1pt, s sep=10pt}
[ [ [] []]   [  [[][]] [[][]] ]  ]
\end{forest}
}}

\begin{document}


\title[Equidecomposable magmas]{Equidecomposable magmas}

\author[C. Card\'o]{Carles Card\'o}
\address{Departament de Ci\`encies de la Computaci\'o, \\Campus Nord, Edifici Omega, Jordi Girona Salgado 1-3. 08034  \\Universitat Polit\`ecnica de Catalunya \\ Barcelona, \\Catalonia (Spain)}
\urladdr{http://www.cs.upc.edu}
\email{cardocarles@gmail.com}

\thanks{This research was supported the recognition 2017SGR-856 (MACDA) from AGAUR (Generalitat de Catalunya).}




\begin{abstract} A magma is called equidecomposable when the operation is injective, or, in other words, if $x+y=x'+y'$ implies that $x=x'$ and $y=y'$. 
A magma is free iff it is equidecomposable and graded, hence the notion of equidecomposability is very related to the notion of freeness although it is not sufficient. We study main properties of such magmas. In particular, an alternative characterization of freeness, which uses a weaker condition, is proved. We show how equidecomposable magmas can be split into two disjoint submagmas, one of which is free.   
Certain tranformations on finite presentations permit to obtain a reduced form which allows us identify all the finite presented equidecomposable magmas up to isomorphisms. 
\end{abstract}

\maketitle


\section {Introduction}
A magma is \emph{equidecomposable} when given two decompositions of an element, $x+y$ and $x'+y'$, we have that they are the same, $x=x'$ and $y=y'$. A magma is free iff it is equidecomposable and graded, \cite{cardo2019arithmetic}. 
Hence,  the notion of equidecomposability is very related to the notion freeness although it is not sufficient. Since this condition is so strong, one could wonder if the totality of such magmas can be identified. We study the main properties of equidecomposable magmas, show characterizations, and finally, we classify up to isomorphisms all the finitely presented equidecomposable magmas. Thus, this article proportions a good perspective of this kind of structures. 

We review some preliminary and necessary concepts in Section~\ref{Preliminaries}, which were already proved in \cite{cardo2019arithmetic}, with some extra result. The article is almost self-contained, with the exception of the proofs of a few lemmas. 
  
In Section~\ref{EquidecMagmas} we define and see first examples of equidecomposable magmas. Section~\ref{GradSemiIniFree} reviews freeness. In particular we obtain an alternative characterization of freeness, by weakening the condition of being graded. A magma is \emph{initial} iff it can be generated entirely by its indecomposable elements. Then, a magma is free if and only if is equidecomposable and initial. 

Another important kind of magmas is that of \emph{full} magmas. A magma is full iff all its elements are decomposable. We show in Section~\ref{FullMagmas} that any equidecomposable magma can be split into two submagmas, the largest full submagma and the initial submagma, as:
\begin{align*}
M = \Ini (M) \sqcup \Full (M),
\end{align*} 
where $\Ini(M)$ and $\Full(M)$ denote respectively the initial part and the largest full magma and $\sqcup$ is the disjoint union of sets.  
  
In section ~\ref{ClosedSubMagmas}, we characterize equidecomposable magmas as a quotient of a free magma by a closed congruence. This allows us to study finite presentations, Section~\ref{Presentations}, and we show how to obtain a fixed form for finite presentations of an equidecomposable magma. More specifically, for each finitely presented equidecomposable magma $M$ there is an injective mapping $\varphi: A \longrightarrow \M_A$, with some additional conditions, such that $M \cong \langle A \mid \varphi \rangle$. We write $\E(\varphi)$ these magmas. 
Using this form we identify up to isomorphisms all the finitely presented equidecomposable magmas. 
We say that two mappings $\varphi: A \longrightarrow \M_A$ and $\psi: B \longrightarrow \M_B$ are \emph{conjugated} when they are the same, up to renaming the generators $A$ and $B$. We will prove in the last Section~\ref{ISoCharac} that two magmas $\E(\varphi)$ and $\E(\psi)$ are isomorphic iff $\varphi$ and $\psi$ are conjugated.

\section {Notation and preliminaries on magmas} \label{Preliminaries}

A \emph{magma} is a set $M$ with an operation $+:M\times M \longrightarrow M$. A \emph{submagma} of a magma $M$ is a subset of $M$ that is stable for the operation. For convenience, we accept the empty set as a magma.
An element $z \in M$ is \emph{decomposable} iff there are $x,y \in \M$ such that $z=x+y$, otherwise it is said \emph{indecomposable}.
We denote $\I=\{0\}$ the \emph{trivial} magma with only one element.  
  
Notions of \emph{lattice, semi-lattice, direct product, homomorphism, isomorphism, monomorphism}, and \emph{epimorphism} of magmas are defined as is usual. 
 See, for example, \cite{Bourbaki1989Algebra, sankappanavar1981course, gratzer2008universal, Rosenfeld1968Algebraic, davey2002introduction} for algebraic elementary concepts. 

\begin{definition} Given a non-empty subset $X$ of a submagma $M$, $\langle X \rangle$ denotes the least submagma containing $X$, that is:
\begin{align*}
\langle X \rangle = \bigcap_{ X \subseteq N} N.
\end{align*}
We say that $X$ \emph{generates} $\langle X \rangle$. Given two submagmas $N,N' \subseteq M$, we write $N \vee N'=\langle N \cup N' \rangle$. 

We also define $\langleD X \rangleD=\bigcup_{k\geq 1} X_k$ where $X_k$ is the sequence: 
\begin{align*}
X_1 &=X;\\
X_k &= \bigcup_{\stackrel{i,j\geq 1}{i+j=k}} X_i  + X_j, \,\, \mbox{ for } k>1.
\end{align*} 
\end{definition}

\begin{lemma} \label{GeneratorSets} Given a non-empty subset $X$ of a submagma $M$, $\langle X \rangle = \langleD X \rangleD$. 
\end{lemma}
\begin{proof} See \cite[p.~33]{sankappanavar1981course}. 
\end{proof}

\begin{definition} \label{DefFreeMag} $\sqcup$ denotes de disjoint union of sets. Given a non-empty set $A$ consider the sequence of sets:
\begin{align*}
X_1 &=A;\\
X_k &= \bigsqcup_{\stackrel{i,j\geq 1}{i+j=k}} X_i  \times X_j, \,\, \mbox{ for } k>1.
\end{align*}
The \emph{free magma on} $A$ is the magma $(\M_A;+)$, where $\M_A=\bigcup_{k\geq 1} X_k$ and $x+y=(x,y)$ for each $x,y\in \M_A$. In general, we say that a magma $M$ is \emph{free} iff there is some set $A$ such that $M\cong \M_A$. The \emph{length} of an element $x\in \M_A$ is defined as $\ell(x)=k$ where $x\in X_k$. Indeed, the length is an epimorphism of magmas, $\ell:\M_A \longrightarrow \Nat$ with $\ell(x+y)=\ell(x)+\ell(y)$, for any $x,y \in \M_A$. 
\end{definition}

\begin{remark} The set $A$ is the unique minimal generator set of $\M_A=\langle A \rangle$. We have adopted the definition of Bourbaki \cite[p.~81]{Bourbaki1989Algebra} which is equivalent to define freeness by satisfying the \emph{universal mapping property}, see \cite[p.~71]{sankappanavar1981course}. 
\end{remark}

\begin{notation} \label{NotationCardo} Following the same notation in \cite{cardo2019arithmetic}, the cyclic free magma is written $\M$ and we take as a generator the symbol $1$. Thus, $\M=\M_{\{1\}}$ and its elements are of the form $1,(1+1), 1+(1+1), (1+1)+1, \ldots$. 
We use the following abbreviations, $2=1+1$, and:
\begin{align*} 
&1_-=1, \,\,\,\,\,\,\,\,\, (n+1)_-=n_-+1,\\
&1_+=1, \,\,\,\,\,\,\,\,\, (n+1)_+=1+n_+,
\end{align*}
for any integer $n\geq 0$.
\end{notation}

A \emph{polynomial} over a magma $M$ is a mapping $P:M^n \longrightarrow M$, defined by the substitution of the \emph{variables} $x_1, \ldots, x_n$ of a term $P \in \M_{\{x_1, \ldots, x_n\}}$ by elements of $M$. Some example of polynomials are $P(x)=x+(x+x)$ or $P(x,y,z)=(x+y)+z$. More formally:

\begin{definition} Let $M$ be a magma and let $P \in \M_Z$, where $Z$ is a finite set called the \emph{variables}. If $\ell(P)=1$,  the \emph{polynomial over $M$} is defined as the \emph{identity mapping} $P: M \longrightarrow M$, $P(z)=z$. If $\ell(P)>1$, then $P$ can be decomposed as $P=Q+R$, where $Q \in \M_Y$, $R \in \M_Z$, and $X\cup Y \subseteq Z$. Let $X=\{x_1, \ldots, x_r\}$, $Y=\{y_1, \ldots, y_s\}$, and $\{z_1, \ldots, z_n\}=X\cup Y$. 
The \emph{polyomial over $M$} is the mapping $P:M^n \longrightarrow M$, defined recursively as:
\begin{align*}
P(z_1, \ldots, z_n)=Q(x_1, \ldots, x_r)+R(y_1, \ldots, y_s).
\end{align*}
\end{definition} 

In a free magma each submagma has a unique minimal generator set which can be calculated as the union of the necessary generators to generate individually each element of the submagma.  

\begin{definition} \label{DefMinGen} Given  a non-empty submagma $N\subseteq \M_A$, consider the mapping $g_N: N \longrightarrow 2^N$ defined recursively as:
\begin{align*}
g_N(z)=\begin{cases} g_N(x) \cup g_N(y) &\mbox{ if } z=x+y \mbox{ for some } x,y\in N; \\
\{z\} & \mbox{ otherwise. }\end{cases} 
\end{align*}
We define $\G(N)=\bigcup_{x\in N} g_N(x)$ when $N$ is non-empty, and $\G(\emptyset)=\emptyset$.
\end{definition}

\begin{lemma}  Given a non-empty submagma $N \subseteq \M_A$, $\G(N)$ is the unique minimal generating set of $N$. We have that if $N$ is a submagma of $\M_A$, then $\langle \G(N) \rangle =N$; and if $X$ is a subset of $\M_A$, then $\G(\langle X \rangle) \subseteq X$. In particular, $\langle \cdot \rangle$ and $\G$ form a Galois connection.
\end{lemma}
\begin{proof} See \cite{cardo2019arithmetic}.
\end{proof}

\begin{lemma} \label{SubmagmaCondition} Given a non-empty submagma $N \subseteq \M_A$ we have that $x+y \in N$ iff either $x,y\in N$ or $x+y\in \G(N)$.
\end{lemma}
\begin{proof} See \cite{cardo2019arithmetic}.
\end{proof}

\begin{definition} Given a magma $M$, a set $X \subseteq M$ is said \emph{closed} iff $x+y \in X$ implies that $x,y \in X$. 
\end{definition}

The intersection and the union of closed sets are closed. Hence, closed sets of a magma $M$ form a complete lattice with bounds $\emptyset$ and $M$. The complement of a closed set is a submagma. Here we are interested in the case that closed sets are in addition submagmas. See \cite{cardo2019arithmetic, Reutenauer1993BookFreeLie} for some questions on closed sets

\section {Equidecomposable magmas} 
\label{EquidecMagmas}

\begin{definition} We say that a magma $M$ is \emph{equidecomposable} iff for each $x,y,x',y' \in M$ we have:
\begin{align*}
x+y=x'+y' \implies x=x', y=y'.
\end{align*} 
Equivalently, $M$ is equidecomposable iff the operation $+: M\times M \longrightarrow M$ is injective. 
We will use the abbreviation \emph{equidec} for equidecomposable magmas.
\end{definition}

\begin{example} \label{ExampTrivial} The trivial magma $\I=\{0\}$ is trivially equidec since the equation $x+y=x'+y'$ has the unique solution $x=y=x'=y'=0$. 
Free magmas $\M_A$  are also equidec. 

Since $|M\times M|=|M|^2>|M|$, provided $|M|>1$, by the pigeonhole principle, the trivial magma is the unique equidec finite magma.

The equidecomposability condition is so strong that there are not commutative, associative, with-neutral-element, nor medial equidec magmas, excepting the trivial magma. Let $M$ be an equidec magma. 
If $M$ is commutative, then for all $x,y \in M$:
\begin{align*}
x+y=y+x \implies x=y  \implies M\cong \I.
\end{align*}
If $M$ is associative, then for all $x,y \in M$:
\begin{align*}
&x+(y+x)=(x+y)+x \\
\implies & x=x+y \mbox{ and } y+x=x \\
 \implies & x+y=y+x.
\end{align*}
Therefore, $M$ is commutative, and then $M\cong \I$.
The cases of magmas with neutral element and medial magmas are similar (see \cite{Jezek1983Medial} for medial magmas). 
\end{example}

\begin{example} We see some constructions of equidec magmas. Since submagmas inherit the operation, if $N$ is a submagma of $M$ and $M$ is equidec, then $N$ is equidec. 
It is also trivial to check that the direct product of two equidec magma is an equidec magma. However the direct product is not in general finitely generated. Consider, for example, the free magma $\M_A \times \M_B$ which needs infinite generators; see later Theorem~\ref{TeoProduct}. We will see that the free product of two finitely presented equidec magmas yields an equidec magma; see later Example~\ref{PropositionProducts}. 

The length $\ell:\M \longrightarrow \Nat$ is an epimorphism of magmas. However, $(\Nat, +)$ is not equidec since it is commutative and associative, see Example~\ref{ExampTrivial}, which shows that the class of equidec magmas is not closed under image of homomorphisms. Equidec magmas do not form a variety but a quasivariety, see \cite[p.~250]{sankappanavar1981course}.
\end{example}

\begin{example} Let us write $\Tree$ the magma of planar binary trees with finite or infinite depth. Given two trees $x,y \in \Tree$, we consider the tree $x+y$ as  the result of join $x$ and $y$ by a new root. Then, $(\Tree, +)$ is an equidec magma.
$\Tree$ is an extension of the cyclic free magma. There is a monomorphism $\Phi: \M \longrightarrow \Tree$, defined by $\Phi(1)=\circ$, where $\circ$ is the tree with only one vertex. For example:
\begin{align*} \Phi \Big( (1+1)+(1+(1+1)) \Big) = \primerexemple \, .\end{align*}
$\Tree_\infty$ denotes the set of infinite trees, and $\Tree_{<\infty}$, the set of finite trees. Both are submagmas and:
\begin{align*}
\Tree_{<\infty}=\Phi(\M), \,\,\,\,\,\, \Tree_\infty=\Tree \setminus \Phi(\M).
\end{align*}
\end{example}

\begin{proposition} \label{Equation} Let $M$ be an equidec magma, $P$ a polynomial over $M$ with $n$ variables, and $k \in M$. The equation:
\begin{align*}
P(x_1,\ldots, x_n)=k,
\end{align*}
has at most one solution. 
\end{proposition}
\begin{proof}
By induction on the length of the polynomial. If $\ell(P)=1$, then $n=1$, and $P(x)=x$. The equation has trivially one solution, $x=k$. 
We assume that the statement is true for polynomials with length less or equal to $m$, with $m>1$ and let $P$ be a polynomial with $\ell(P)=m+1$. We suppose that $k$ can be decomposed as $k=k_1+k_2$, otherwise the equation has no solution and we are done. Since $\ell(P)>1$, we can decompose $P$ as $P(x)=P_1(x)+P_2(x)$, where $x=(x_1,\ldots, x_n)$. Since $M$ is equidec:
\begin{align*}
P(x)=k \implies P_1(x)+P_2(x)=k_1+k_2 \implies P_1(x)=k_1, P_2(x)=k_2.
\end{align*}
Equations $P_i(x)=k_i$, $i=1,2$ have length less than $m+1$, and by hypothesis of induction these equations has at most one solution. If both equations $P_i(x)=k_i$, $i=1,2$ have solution, the equation $P(x)=k$ has a unique solution; if one of them does not have solution, the equation does not have solution.
\end{proof}

\begin{example} Let $P(x,y)$ be a polynomial over the equidec magma $(M;+)$. The magma $(M; \odot)$ with $x\odot y =P(x,y)$ is equidec.  
This is a straightforward consequence of Proposition~\ref{Equation}. 
\end{example}

\begin{example} \label{ExampLan} Let $\Sigma=\{\alpha, \beta\}$ be an \emph{alphabet} and let $(\Sigma^*; \cdot, \varepsilon)$ be the free monoid on $\Sigma$. We write $\Lan=2^{\Sigma^*}$ the set of \emph{languages} on the alphabet $\Sigma$. See \cite{hopcroft1979introduction} for terminology on formal languages. We adopt the usual convention of notating singleton languages $\{x\}$ as $x$. 

Given, $L,L'\in \Lan$, consider the operation:
\begin{align*}
L \oplus L'=\alpha L \cup \beta L',
\end{align*}
and let us see that $(\Lan; \oplus)$ is a non-countable equidec magma. First we see an elementary result on sets. Let $X,Y,Z,T$ be sets such that:
\begin{align*}
X\cap Y=X\cap T=Z \cap Y=Z \cap T=\emptyset.
\end{align*}
We have that:
\begin{align*}
X\cup Y =Z\cup T \implies X=Z, Y=T.
\end{align*}
Let us see a quick proof: 
\begin{align*}
X=X\cap(X\cup Y)=X \cap (Z \cup T)=(X\cap Z)\cup \emptyset=X \cap Z,
\end{align*}
whence $X\subseteq Z$. By symmetry we get $X=Z$. Similarly we get $Y=Z$. 
Now we can see that the operation $\oplus$ is equidec. We suppose: $L \oplus L'=H \oplus H'$, for some $L,L',H,H' \in \Lan$. This means: $ \alpha L \cup \beta L'=\alpha H \cup \beta H'$. Clearly $\alpha L\cap \beta H'=\beta L' \cap \alpha H=\emptyset$ and $\alpha L \cap \beta L'= \alpha H \cap \beta H'=\emptyset$. Then, applying the above result, $\alpha L =\alpha H$ and $\beta L' = \beta H'$. Therefore, $L=H$ and $L'=H'$. 

Let us see some submagmas in $\Lan$. A language $L$ is \emph{prefix free}, when $xy \in L$ implies that $x \not \in L$. The set of prefix free languages is a submagma. 
Many classical families of languages, such as finite languages, \emph{regular} languages, \emph{context-free} languages, \emph{context-sensitive} languages, \emph{recursive} languages, amongst others, are submagmas. For example, the family of regular languages is closed, amongst other operations, under products and unions; see \cite{hopcroft1979introduction}. Thus, if $L,L'$ are regular, then $\alpha L\cup \beta L'$ is. 

There is a copy of $\Tree$ in $\Lan$. Given a binary tree we label left branches with $\alpha$, and right branches, with $\beta$. The path from the root to a vertex is a word in $\Sigma^*$. We define the monomorphism $\Psi:\Tree \longrightarrow \Lan$ where $\Psi(x)$ is the set of all the word paths of the leaves of the tree $x$. We have indeed:
\begin{align*}
\M \cong \Psi(\Tree_{<\infty})=\langle \varepsilon \rangle.
\end{align*}
\end{example}

\begin{example} A \emph{left semiringoid} structure $(M; +, \cdot, 1)$ is a magma $(M;+)$ such that, for any $x,y,z \in M$:
\begin{align*}
& x\cdot 1=1 \cdot x=x,\\
& (x\cdot y)\cdot z=x \cdot (y \cdot z),\\
& x\cdot (y+z)=x\cdot y + x \cdot z.
\end{align*}
A \emph{right semiringoid} is defined symmetrically. 
Given $a\in M$, the set $a \cdot M$ is a submagma of $M$, which we call \emph{principal} ($M\cdot a$ for the case of right semiringoids). 
A homomorphism of semiringoids preserves the sum, the unity and the product. An anti-homomorphism preserves the sum, the unity, and reverses the product order. 

The free magma $(\M; +, \cdot,1)$ is a left semiringoid, where the product $x\cdot y$ consists in substituting each $1$ of $y$ by $x$. See \cite{blondel1995UneFamille, blondel1994Properties} for operations defined over a free magma. In \cite{cardo2019arithmetic} we studied the arithmetical properties of this structure. Every cyclic submagma of $\M$ is principal, $\langle a \rangle= a\M$. 

The product of languages is distributive with $\oplus$:
\begin{align*}
(L\oplus L') S=(\alpha L \cup \beta L')S= \alpha LS \cup \beta L'S=L S \oplus L' S.\end{align*}
In addition: $ L \cdot \varepsilon=\varepsilon \cdot L=L$. Thus, $(\Lan; \oplus, \cdot, \varepsilon)$ is a right semiringoid. 

Let us compare the semiringoids $\Lan$ and $\M$. 
We have that $\langle \varepsilon \rangle$ and $\langle 1 \rangle=\M$ are anti-isomorphic ringoids. That is, both are the same up to the order of the factors in the product operation. However, $\Lan$ is much complex magma. Cyclic submagmas of $\Lan$ are not principal. For example:
$\langle \varepsilon \rangle \subsetneq \Lan \cdot \varepsilon=\Lan$.
Finally, $\M$ with the product, $(\M; \cdot, 1)$, is a non-finitely generated free monoid, while $(\Lan; \cdot, \varepsilon)$ has, for instance, the non-trivial relation $\beta^* \cdot \beta^* =\beta^*$.

The magma $\Tree$ can also enriched to obtain a ringoid. We define the product of two trees $x \bullet y$ as the result of hanging the tree $x$ from each leave in $y$.  For instance:
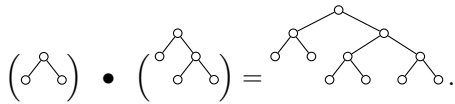
\begin{align*} \Big(\dos \Big) \,\,\, \bullet \,\,\, \Big( \,\trespositiu \Big)= \producte \, .\end{align*}
\end{example}

\section{Graded, semigraded, initial and free magmas}
\label{GradSemiIniFree}

Let $(\Nat;+)$ the set of positive integers with the ordinary sum. A \emph{gradation} of a magma $M$ is a homomorphism $\ell: M \longrightarrow \Nat$, so that $\ell(x+y)=\ell(x)+\ell(y)$. The length is a gradation for free magmas. 
In \cite{cardo2019arithmetic} we proved that a magma is free iff it is equidec and graded.

\begin{example} The condition of being equidec is necessary. For example, the commutative monoid $(\Nat; +)$ is graded with $\ell(x)=x$, but it is not a free magma. 

The condition of being graded is necessary. 
$\Tree$ is not free because we have, for example, the relation given by the infinite full binary tree $t$, $t=t+t$. This relation disallows to define a gradation since 
$\ell(t)=\ell(t)+\ell(t)$ implies that $\ell(t)=0$, which is impossible since gradations must be positive.
\end{example}

\begin{example} 
$\Lan$ is not free since, for example, we have the relation $\emptyset \oplus \emptyset=\emptyset$, or for example:
\begin{align*}
\varepsilon \oplus \beta^*\alpha=\alpha \cup \beta\beta^*\alpha=\beta^*\alpha.
\end{align*}
Consider the set of finite languages $\Fin \subset \Lan$. Notice that $\langle \varepsilon \rangle \subsetneq \Fin$. Since $\alpha L \cap \beta L'=\emptyset$, $|L\oplus L'|=|L|+|L'|$. Then we have the natural gradation $\ell(L)=|L|$, for which, in addition, $\ell(LL')=\ell(L)\ell(L')$.  
$\Fin$ is clearly submagma of $\Lan$. Since it is graded, it is free. However, it is not finitely generated. Singleton languages $L$ cannot be decomposed into other languages by $\oplus$ (otherwise, if $L=L' \oplus L''$, then $1=\ell(L)=\ell(L')+\ell(L'')>1$). This means that any singleton language must be a generator. Since we have infinite singleton languages $\Fin$ is not finitely generated. 
\end{example}

\begin{remark} \label{ReIndecomposable} We notice that for any generator set $A$ of a magma $M$, we have that $M\setminus(M+M) \subseteq A$. Thus, the set $M \setminus (M+M)$ is always involved to generate $M$.  
\end{remark}

\begin{definition} We define the \emph{initial part} of a magma $M$ as the submagma genererated by the indecomposable elements, or what is the same:
\begin{align*}
\Ini(M)=\langle M \setminus (M+M)\rangle.
\end{align*}
We say that $M$ is \emph{initial} when $\Ini(M)=M$. 
\end{definition}

The condition of being graded can be substituted by the weaker condition of being initial.

\begin{theorem} \label{TeoInitEqui} A magma is free iff it is initial and equidec. 
\end{theorem} 
\begin{proof} $(\Rightarrow)$ Trivial, since free magmas are initial. $(\Leftarrow)$ Let $M$ be a initial and equidec magma. We take the sequence:
\begin{align*}
X_1 &=M \setminus (M+M);\\
X_k &= \bigcup_{\stackrel{i,j\geq 1}{i+j=k}} X_i+X_j.
\end{align*}
We are going to see that the succession $\{X_k\}_{k\geq 1}$ forms a partition of $M$. On the one hand, by Lemma~\ref{GeneratorSets}, we have: 
\begin{align*}
\bigcup_{k\geq 1} X_k=\langle \! \langle M \setminus (M+M) \rangle \! \rangle=\langle M \setminus (M+M) \rangle.
\end{align*}
Since $M$ is initial $\langle M\setminus (M+M)\rangle=\Ini (M)=M$. 
On the other hand we need to prove that $X_i\cap X_j=\emptyset$, for any $i\not=j$.
First we consider the cases $X_1\cap X_j$ and $X_i\cap X_1$, provided $i,j\not=1$. We have $X_1\cap X_j= \emptyset$, since $X_1=M\setminus (M+M)$ is the set of indecomposable elements, while elements in $X_j$ are clearly decomposable by definition. For the same reason $X_i\cap X_1=\emptyset$. 

Now we prove that if $X_i\cap X_j\not =\emptyset$ with $i,j>1$, then there are some $i'<i$ and $j'<j$ such that $X_{i'}\cap X_{j'}\not =\emptyset$. 
If $X_i\cap X_j\not =\emptyset$, then there must be an element such that $x \in X_i$ and $x\in X_j$. Since $i,j>1$, $x$ can be decomposed into two forms:
\begin{align*}
& x=x_{i'}+x_{i''}, \mbox{ with } i=i'+i'', x_{i'}\in X_{i'}, x_{i''}\in X_{i''};\\
& x=x_{j'}+x_{j''}, \mbox{ with } j=j'+j'', x_{j'}\in X_{j'}, x_{j''}\in X_{j''}.
\end{align*}
Since $M$ is equidec:
\begin{align*}
x=x_{i'}+x_{i''}=x_{j'}+x_{j''} \implies x_{i'}=x_{j'} \mbox{ and } x_{i''}=x_{j''}.
\end{align*}
And this means that $X_{i'}\cap X_{j'}\not=\emptyset$ and $X_{i''}\cap X_{j''}\not=\emptyset$. Since $i=i'+i''$, $j=j'+j''$, and $i',i'',j',j''>1$, we have that $i'<i$ and $j'<j$.

By reduction to the absurd, we suppose that $X_i \cap X_j \not=\emptyset$. By applying sufficiently times the last result, we will get that there must be subscripts $i',j'$ such that either $X_1 \cap X_{j'}\not=\emptyset$ or   $X_{i'}\cap X_1 \not=\emptyset$, which is a contradiction. 
 
Finally, we define recursively the following succession of mappings:
\begin{align*}
f_1(z)& =z  \,\,\,\,\,\,\,\,\,\,\,\,\,\,\,\,\,\,\,\,\,\,\,\,\,\,\,\,\,\,\,\,\,\,\,\, \mbox{ if } z\in X_1, \\ 
f_k(z)&=f_i(x)+f_j(y)  \,\,\,\,\, \mbox{ if } z\in X_k,
\end{align*}
where $z=x+y$, $x \in X_i$ and $y \in X_j$ and $k=i+j$. Since the decomposition $x+y$ is unique and since $\{X_k\}_{k\geq 1}$ forms a partition, $i$ and $j$ are determined uniquely and $f_k$ is well-defined. 
If we join all these mappings:
\begin{align*}
f=\bigsqcup_{k\geq 1} f_k,
\end{align*}
we get a mapping $f: M \longrightarrow \M_X$ and we only need to see that it is an isomorphism of magmas. $f$ is trivially a homomorphism. In addition, since $\langleD M \setminus (M+M)\rangleD=M$, $f$ is surjective. Injectivity can be proved by induction on $k$.
\end{proof}

\begin{definition} We say that a magma $M$ is semigraded iff there is a mapping $\mu: M \longrightarrow \Nat$ such that: 
\begin{align*}
\mu(x+y)>\mu(x), \mu(y).
\end{align*}
\end{definition}

\begin{proposition} \label{PropImplications} Given a magma $M$, we have the implications:
\begin{align*}
M \mbox{ is graded } \implies M \mbox{ is semigraded} \implies M \mbox{ is initial}.
\end{align*}
\end{proposition}
\begin{proof} The first implication is trivial. We see the second one. By definition, $\Ini(M) \subseteq M$. We see that if $M$ is semigraded with the semigradation $\mu$, then $M \subseteq \Ini(M)$.
Let $z\in M$, and we see by induction on the number $\mu(z)$ that $z\in \Ini(M)$. 
If $\mu(z)=1$, then $z$ is not decomposable, otherwise, $z=x+y$ implies $1=\mu(z)>\mu(x),\mu(y)$ is not possible, since semigradations are positive.  
This means that $z \in M\setminus (M+M) \subseteq \Ini(M)$. 

Let us suppose that the statement is true for any $x$ such that $\mu(x)<n$, with $n>1$. Let $z\in M$ such that $\mu(z)=n$.   If $z$ is not decomposable, then we are done, as above. So, let $z=x+y$. Since $\mu(x),\mu(y)<n$, by hypothesis of induction $x,y\in \Ini(M)$, and then $z=x+y \in \Ini(M)$.  
\end{proof}

\begin{theorem} \label{CoroEquiva} Let $M$ be a magma. We have the following equivalences:
\begin{enumerate}
\item $M$ is equidec and graded;
\item $M$ is equidec and semigraded;
\item $M$ is equidec and initial;
\item $M$ is free.
\end{enumerate}
\end{theorem}
\begin{proof} Implications $(1) \Rightarrow (2) \Rightarrow (3)$ are consequence of Proposition~\ref{PropImplications}.  By Theorem~\ref{TeoInitEqui}, $(3) \Rightarrow (4)$. The last implication $(4) \Rightarrow (1)$ is trivial. 
\end{proof}

\begin{example} Consider the magma $\Nat_\uparrow=(\Nat; \uparrow)$,  where:
\begin{align*}
x\uparrow y=2^x\cdot 3^y.
\end{align*}
By the uniqueness of decomposition in prime numbers, the magma is equidec. It can be graded, but there is not a closed formula. However, it is semigraded by the identity mapping $\mu(x)=x$, and, by Theorem~\ref{CoroEquiva}, it is free. 
This magma is not finitely generated. For instance, any number in $\Nat \setminus 6 \Nat$ is, amongst others, an indecomposable element under $\uparrow$ operation. 
\end{example}

\begin{example} The condition of being initial is strictly weaker than the condition of being graded. Consider, for example, the magma $M=\{a,b\}$ with the constant operation:
$ x + y= b$, for any $x,y \in M$. 
$M$ is an initial magma $\Ini(M)=\langle a \rangle = M$, but it cannot be graded nor semigraded. 
\end{example}

\section{Full magmas and a decomposition theorem}
\label{FullMagmas}

\begin{definition} We say that a magma is \emph{full} iff every element has a decomposition in the magma.   
\end{definition}

\begin{example} The trivial magma $\I$ is full. Any semilattice $(L; \wedge)$ is full by doing $x=x\wedge x$, for any $x\in L$. Any group, actually any monoid,  $(G;\cdot)$ is full by doing $g=1\cdot g$, for any $g\in G$. 
The magmas $\M_A$, $\Tree$, and $\Lan$ are not full, with the counterexamples: any generator $a\in A$, the tree with only one vertex $\circ$, and the language $\varepsilon$, respectively. 
\end{example}

\begin{example} \label{ExampleTrivialFull} 
There are some equivalent ways to say that a magma is full. $M$ is full iff the the operation is surjective. $M$ is full iff $M+M=M$. Even, $M$ is full iff  $\Ini(M)=\emptyset$. 
In particular, semigraded magmas are not full. Notice that in order to see if a magma $M=\langle A \rangle$  is full we only have to inspect the generators, since elements in $M \setminus A$ are trivially decomposable.    

Let $M,N$ be magmas. If $M,N$ are full, then $M\times N$ is full, and the free product $M*N$ is full. If $f: M \longrightarrow N$ is a homomorphism and $M$ is full, then $f(M)$ is full. 
\end{example}

\begin{example} \label{ExampDiamond} Let us consider a non-trivial example of full and equidec magma. A magma is equidec and full iff the operation is injective and surjective, that is, $+:M\times M \longrightarrow M$ is a bijection.
A classic geometric construction proves that $\Nat$ and $\Nat^2$ are equinumerous, see \cite{enderton1977elements}:
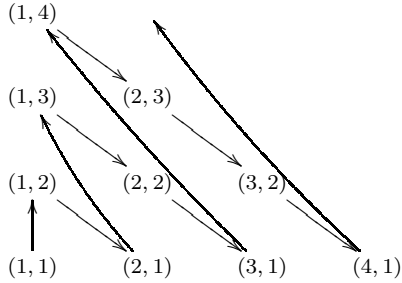
\begin{align*}
 \footnotesize
  \xymatrix@C-0.4pc@R-0.4pc{(1,4) \ar[dr] &&&\\
(1,3) \ar[dr] & (2,3) \ar[dr] &&\\
(1,2) \ar[dr] & (2,2) \ar[dr] & (3,2) \ar[dr] &\\
(1,1)  \ar[u]& (2,1) \ar@/^0.3pc/[uul] & (3,1) \ar@/^0.2pc/[uuull] & (4,1) \ar@/^0.3pc/[uuull] }
\end{align*}
There is a closed expression for this bijection which defines an operation:
\begin{align*}
x \diamond y= \frac{1}{2} ( x^2+y^2+2xy -x-3y+2  ).
\end{align*}
Therefore, $\Nat_\diamond=(\Nat; \diamond)$ is an equidec and full magma.  
We notice that we have the relations: 
\begin{align*}
1\diamond 1=1 \,\,\,\, \mbox{ and } \,\,\,\, 2=1 \diamond 2.
\end{align*}
By using the following property:
\begin{align*}
(x+1) \diamond (y-1) =(x \diamond y )+1,
\end{align*}
provided that $y>1$, one can prove by induction that $\langle 1,2 \rangle =\Nat$. 
In addition, we have that for any $x,y\geq 3$, $x \diamond y>x,y$, disallows any other non-trivial relation. Actually, the identity $\mu(x)=x$ is a semigradation for the submagma $\Nat\setminus \{1,2\}$, which means that it is free by Theorem~\ref{CoroEquiva}. 
\end{example}

\begin{example} \label{JonssonTarskiAlg} A \emph{J\'onsson-Tarski algebra}, or a \emph{Cantor algebra}, is an algebra $(M; +, p,q)$ of type (2,1,1), that is, a binary operation $+$ and two unary operations $p,q$, such that:
\begin{align*}
p(x+y)=x, \,\,\,\, q(x+y)=y, \,\,\,\, p(z)+q(z)=z,
\end{align*}
for each $x,y,z \in M$. See \cite{Jonsson1961Properties}. It turns that, given an algebra  $(M;+,p,q)$ of type (2,1,1), $(M;+,p,q)$ is a J\'onsson-Tarski algebra iff $(M;+)$ is a full and equidec magma. The operations $p$ and $q$ are just the first and the second component of the decomposition of each element. Given a full and equidec magma $M$ we write $\widetilde{M}$ for its J\'onsson-Tarski algebra.

By taking the three operations $+,p,q$, finitely generated J\'onsson-Tarski algebras are always cyclic. If $A\cup \{x,y\}$ generates a J\'onsson-Tarski algebra, where $x,y \not  \in A$, then $A\cup \{x+y\}$ equally generates the algebra. Thus, by induction, one can reduce the generators to a single one. 
For example, $\widetilde{\Nat}_\diamond$ is generated by $2$ as J\'onsson-Tarski algebra but we need two generators $1,2$ as the magma $\Nat_\diamond$.
See Example~\ref{ExempNonEmagma} and last Remark~\ref{LastJonssonTarski} for more aspects on this class of algebras.   
\end{example}

\begin{example} \label{ExamSeq} Let $\Seq(X)$ be the set of infinite sequences over a finite non-empty set $X$. Consider the ``shuffle'' operation:
\begin{align*}
(a_1, a_2, a_3, \ldots) \curlyvee (b_1, b_2, b_3, \ldots)=(a_1,b_1,a_2, b_2, a_3, b_3, \ldots).
\end{align*}
$(\Seq; \curlyvee)$ is equidec and given a sequence we always can decompose it into two sequences:
\begin{align*}
(a_1, a_2, a_3,a_4,a_5,a_6 \ldots)=(a_1,a_3,a_5,\ldots) \curlyvee (a_2, a_4, a_6, \ldots).
\end{align*}
Hence, $\Seq(X)$ is a full magma. 
When $|X|=1$,  $\Seq(X) \cong \I$. Let us see some submagmas in $\Seq(X)$:
\begin{enumerate}
\item Of course, if $X \subseteq Y$, we have the submagma $\Seq(X) \subseteq \Seq(Y)$. 

\item Let $X=\{0, \alpha_1, \ldots, \alpha_n\}$. We have:
\begin{align*}
\langle (\alpha_1, 0,0,0, \ldots), \ldots, (\alpha_n, 0,0,0, \ldots)\rangle \cong \M_{X\setminus \{0\}}.
\end{align*}

\item We say that a sequence $(a_n)$ is \emph{periodic} iff there is a positive integer $T$, called a \emph{period}, such that: 
\begin{align*}
n \equiv m \mod T \implies a_n=a_m.
\end{align*}
$\tau(a_n)$ denotes the least of the periods of $a_n$, provided it is periodic. 
If $a_n$ and $b_n$ are periodic sequences, then $a_n \curlyvee b_n$ is periodic and:
\begin{align*}
\tau ( a_n \curlyvee b_n)\leq 2\cdot \lcm (\tau(a_n), \tau (b_n)).
\end{align*}
Therefore, the set of periodic sequences $\Per(X) \subset \Seq(X)$ is a submagma.  
\end{enumerate}
\end{example}

\begin{remark} For any arbitrary family of full submagmas $N_i \subseteq M$, $i\in I$ for some set of subscripts $I$, the submagma $H=\bigvee_{i \in I} N_i$ is full. As we commented in Example~\ref{ExampleTrivialFull}, we only need to see whether the generators of a generating set of $H$ are decomposable, for example, $\bigcup_{i \in I} N_i$. Since the elements in $N_i$ are decomposable in $N_i$ for each $i \in I$, $H$ is full.   
In other words, there exists always the largest full submagma of a given magma.
Notice that, in general, the intersection of full magmas is not longer full. And in particular, a full magma can contain non-full submagmas. 
\end{remark}
\begin{notation} Let $M$ be a magma. We write: 
\begin{align*}
\Full(M)=\bigvee_{N \subseteq M \mbox{ full}} N,
\end{align*}
the largest full submagma in $M$. Thus, $M$ is full iff $\Full(M)=M$. 
\end{notation}

\begin{lemma} \label{LemmaEquiFullClosed} The full submagmas of a given equidec magma are closed.
\end{lemma}
\begin{proof} Let $N$ be a submagma of a magma $M$, and let $z=x+y\in N$. Since $N$ is full, $z$ can be decomposed $z=x'+y'$ with $x',y' \in N$. Since $M$ is equidec, $N$ is equidec, and then $x=x', y=y'$. Thus, $x,y \in N$, and hence, $N$ is closed.     
\end{proof}

\begin{theorem} \label{TheoSplit} Each equidec magma $M$ can be split as:
\begin{align*}
M= \Ini(M) \sqcup \Full(M), 
\end{align*}
where $\sqcup$ is the disjoint union of sets. In addition, the initial part is free.
\end{theorem}
\begin{proof} We consider the decomposition: $M=\Ini(M) \sqcup (M\setminus \Ini(M))$. First we show that $M\setminus \Ini(M)$ is a submagma of $M$. 
Let $x,y \in M \setminus \Ini(M)$, which means that $x$ and $y$ cannot be generated by indecomposable elements. Let us suppose that $x+y \in \langle a_1, \ldots, a_n \rangle$ for some $a_1, \ldots, a_n \in M \setminus (M+M)$. Then, there is a polynomial $P$ over $M$ such that $x+y=P(a_1, \ldots, a_n)$. If $\ell(P)=1$, then $n=1$ and $x+y=a_1$, which means that $x+y$ is indecomposable, which is clearly absurd. If $\ell(P)>1$, then there are two polynomials $Q,R$ such that $P=Q+R$, and then $x+y=Q(a_1, \ldots, a_n)+ R(a_1, \ldots, a_n)$. Since $M$ is equidec, we have that $x=Q(a_1, \ldots, a_n)$ and $y=R(a_1, \ldots, a_n)$, whereby $x,y$ should be generated by initial elements, which is a contradiction. This proves that $M \setminus \Ini(M)$ is a submagma. 

Now we see that $M \setminus \Ini(M)$ is full. Let $z\in M\setminus \Ini(M)$ and we suppose that $z$ is indecomposable. Then $z\in M\setminus (M+M) \subseteq \Ini(M)$,  which is clearly a contradiction.

Finally we see that for any other full submagma $N$ we have $N \subseteq M\setminus \Ini(M)$. Let us suppose that $N \not \subseteq M \setminus \Ini(M)$. This means that $N \cap \Ini(M)\not = \emptyset$. Let $x \in N \cap \Ini(M)$. $x$ should be generated by indecomposable elements. That is, there is a polynomial $P$ and indecomposable elements $a_1, \ldots, a_n$ such that $x=P(a_1, \ldots, a_n)$. By Lemma~\ref{LemmaEquiFullClosed}, $N$ is closed, which means that $a_1, \ldots a_n \in N$.  However, $N$ is full by hypothesis and their elements must be decomposable which is a contradiction. This means that $M \setminus \Ini(M)$ is the largest full submagma, $\Full(M)=M \setminus \Ini(M)$. 

The initial part of a magma is a initial magma. By Theorem~\ref{TeoInitEqui}, $\Ini(M)$ is free. 
\end{proof}

\begin{example} Theorem~\ref{TheoSplit} provides a first way to differentiate equidec magmas, since isomorphic magmas must have isomorphic parts. Let us calculate them for the known magmas, which demonstrates that they are quite different structures:
\begin{enumerate}
\item Free magmas, $\M_A$, does not have full part. That is:
\begin{align*}
\Ini(\M_A)=\M_A, \,\,\,\,\,\,\,\,\,\, \Full(\M_A)=\emptyset.
\end{align*}

\item The initial part of the magma of trees, $\Tree$, coincides with the submagma of finite trees, while the full part coincides with the submagma of infinite trees: 
\begin{align*}
\Ini(\Tree)=\Tree_{<\infty}\cong \M, \,\,\,\,\,\,\,\,\, \Full(\Tree)= \Tree_{\infty}.
\end{align*}

\item The magmas $\Nat_\diamond$ and $\Seq(X)$ are full, so these magmas have no initial part:
\begin{align*}
\Ini(\Nat_\diamond)=\emptyset,& \,\,\,\,\,\,\,\,\,\,\,\,\,\,\,\, \Full(\Nat_\diamond)= \Nat_\diamond, \\
\Ini(\Seq(X))=\emptyset,& \,\,\,\,\,\,\,\,\,\,\,\,\,\,\,\, \Full(\Seq(X))= \Seq(X).
\end{align*}
We remark that the initial part of an equidec magma is free, but it is not necessarily  the largest free submagma in $M$. In the Example~\ref{ExamSeq}, we saw that $\Seq(X\sqcup\{0\})$ contains the free magma $\M_X$, whereas $\Ini(\Seq(X))=\emptyset$.

\item In the case of the magma of languages $\Lan$, we have that:
\begin{align*}\Ini(\Lan)=\langle \varepsilon \rangle \vee \langle \alpha \Lan \cup \beta \Lan \rangle.
\end{align*}
Let us see it. Given $x\in \Sigma^*$, let $\First(x)$ be the first letter of $x$, when  $x\not=\varepsilon$, and $\First(\varepsilon)=\varepsilon$. Then:
\begin{align*}
 & \Lan\oplus \Lan =\{ L \in \Lan \mid \First(L)=\Sigma\}\\
\implies & \Lan \setminus (\Lan\oplus \Lan)=\{ L \in \Lan \mid \First(L)\not=\Sigma\}\\
&=\{L \in \Lan \mid \First(L)=\{\varepsilon\}, \{\alpha\}, \mbox{ or }\{\beta\} \}\\  
 \implies & \Ini(\Lan)=\langle \Lan \setminus (\Lan\oplus \Lan)\rangle=\langle \{\varepsilon\} \cup \alpha \Lan \cup \beta \Lan \rangle\\
 &= \langle \varepsilon \rangle \vee \langle \alpha \Lan \cup \beta \Lan \rangle.
\end{align*}
\end{enumerate} 
\end{example}

The direct product of equidec magmas behaves in a different way to other algebraic structures, such as monoids or groups. For example, there are not the inclusion monomorphisms $M,N \longrightarrow M \times N$. Recall that an equidec magma does not have a neutral element, but for the trivial magma, see Example~\ref{ExampTrivial};  and, in general, equidec magmas need not have idempotent elements. A bit more surprisingly, an equidec magma by a free magma is a free magma: 
\begin{theorem} \label{TeoProduct} If $M$ is an equidec magma, then:
\begin{align*}\M_A \times  M \cong \M_{(A \times M) \cup (\M_A \times \Indec(M)) },
\end{align*}
where $\Indec(M)=M\setminus (M+M)$ is the set of indecomposable elements of $M$.
\end{theorem} 
\begin{proof} We define the gradation $\tilde{\ell}(x,y)= \ell(x)$, where $\ell(x)$ is the gradation of the free magma $\M_A$. Since the product of equidec magmas is equidec, $\M_A \times  M$ is a free magma. Now notice that the indecomposable elements are:
\begin{align*}\Indec (\M_A\times M)=(A \times M) \cup (\M_A \times \Indec(M)).\end{align*}
Since the product is a free magma, it is an initial magma generated by these elements. Then, the isomorphism is defined over the generators by the identity.
\end{proof}

\begin{example}  Some particular cases of Theorem~\ref{TeoProduct} are the following:
\begin{enumerate}
\item  When $M=\I$, $\Indec(M)=\emptyset$, and we have: 
\begin{align*}
\M_A \times \I=\M_{(A \times \I) \cup (\M_A \times \emptyset)}\cong \M_{A \times \I}\cong \M_A. 
\end{align*}
\item If $M$ is full, then $\Indec(M)=\emptyset$, whereby:
\begin{align*} \M_A \times  M \cong \M_{ A \times M }.
\end{align*}
\item When $M=\M_B$, $\Indec(\M_B)=B$, and we have:
\begin{align*} 
\M_A \times \M_B \cong \M_{(A \times \M_B) \cup (\M_A \times B)}.
\end{align*}
Notice that:
$$\M \not \cong \M^2 \cong \M^3 \cong \M^4 \cong \cdots \cong \M_\omega,$$
where $\omega$ is any countably infinite set. 
\end{enumerate}
\end{example}

\section{Closed submagmas of a free magma and quotient characterization}
\label{ClosedSubMagmas}

\begin{notation} Given a set $X$, $\Delta(X)=\{ (x,x) \mid x \in X\}$ denotes the \emph{diagonal set}. Let $M$ be a magma.  
 A \emph{congruence on $M$} is an equivalence relation such that it is in addition a submagma of $M^2$. 
Given a subset $X \subseteq M^2$, $\Xi(X)$ denotes the least \emph{equivalence relation} containing $X$. 
$\Theta(X)$ denotes the least congruence in $M^2$ containing  $X$. Thus, $\Xi(X) \subseteq \Theta(X)$, and $\langle X \rangle \subseteq \Theta(X)$. Given a congruence $\theta$ on $M$, $[x]=\{ y \mid (x,y) \in \theta \}$ is the \emph{equivalence class} of the element $x\in M$. The \emph{natural projection} is the epimorphism $\pi: M \longrightarrow  M/\theta$, $\pi (x)=[x]$.  
\end{notation}

We remark that congruences on $\M_A$ are subsets of:
\begin{align*}
\M_A^2 \cong \M_{(A\times \M_A) \cup (\M_A \times A)},
\end{align*}
which is a non-finitely generated free magma, recall Theorem~\ref{TeoProduct}.
$\G(N)$ denotes the unique minimal generator set of a submagma $N \subseteq \M_A$, Definition~\ref{DefMinGen}.  
\begin{lemma} \label{LemaClosedMagmas} Let $N$ be a submagma of a free magma $N \subseteq \M_A$. We have that $N$ is closed iff $\G(N) \subseteq A$.
In particular, the unique closed submagmas of the free magma $\M_A$ are $\M_B$ with $B\subseteq A$.
\end{lemma}
\begin{proof} 
 $(\Rightarrow)$  Let $z\in N$ and suppose that $z$ can be decomposed $z=x+y$. Since $N$ is closed, $x,y \in N$. Hence $z$ is not a generator. This means that the generators of $N$ are those that cannot be decomposed, that is, the elements in $A$.

$(\Leftarrow)$ By Lemma~\ref{SubmagmaCondition} if $x+y \in N$ then either $x,y \in N$ or $x+y \in \G(N) \subseteq A$. But this last one is not possible since generators in $A$ are indecomposable. therefore, $N$ is closed.  
\end{proof}

\begin{lemma} \label{LemaEquiClosed} Let $f: M \longrightarrow N$ be a homomorphism of magmas. Then, $f(M)$ is equidec iff $\ker f$ is a closed submagma of $M^2$.
\end{lemma}
\begin{proof} We have the following equivalences, where the statements must be read with the prefix ``for any $x,y,x',y' \in M$'':
\begin{align*} 
 & f(M) \mbox{ is equidec}\\
\iff & \mbox{ if } f(x)+f(x')=f(y)+f(y'), \mbox{ then } f(x)=f(y),f(x')=f(y')\\
\iff & \mbox{ if } f(x+x')=f(y+y'), \mbox{ then } f(x)=f(y),f(x')=f(y')\\
\iff & \mbox{ if } (x+x',y+y') \in \ker f, \mbox{ then } (x,y),(x',y') \in \ker f\\
\iff & \mbox{ if } (x,y)+ (x',y') \in \ker f, \mbox{ then } (x,y),(x',y') \in \ker f\\
\iff & \ker f \mbox{ is closed}. \qedhere
\end{align*}
\end{proof}

\begin{theorem} \label{CaracEquidec} Let $\theta \subseteq \M_A
^2$ be a congruence on  $\M_A$. The following statements are equivalent:
\begin{enumerate}
\item $\M_A /\theta$ is equidec;
\item $\theta$ is closed;
\item $\G(\theta)\subseteq (A\times \M_A) \cup (\M_A \times A)$.
\end{enumerate}
\end{theorem}
\begin{proof} We consider the epimorphism $\pi: \M_A \longrightarrow \M_A/\theta$.  By the \emph{isomorphism theorem} for magmas, \cite[p.~47]{Bourbaki1989Algebra}, we have $\ker \pi= \theta$. Applying Lemma~\ref{LemaEquiClosed}, we get the equivalence $(1) \Leftrightarrow (2)$.

The equivalence $(2) \Leftrightarrow (3)$ follows from Lemma~\ref{LemaClosedMagmas} and from the fact that congruences are submagmas of the free magma $\M_A^2$, with $\G(\M_A^2)=(A\times \M_A) \cup (\M_A \times A)$.
\end{proof}

\section{Presentations}
\label{Presentations}

Given a presentation of a group, the so-called \emph{Tietze transformations} allow to add and remove superfluous generators and relations without changing the group. 
Tietze transformations are possible also in magma theory. In addition in the category of equidec magmas there is an extra general transformation: we can substitute a relation by its generators in $\G(\M_A^2)$, since congruences on equidec magmas must be closed.  

Since the intersection of closed sumagmas is a closed submagma, it makes sense to define the least closed congruence. 

\begin{definition} Given a magma $M$, and a subset $X\subseteq M^2$, $\boxminus(X)$ denotes the least closed congruence on $M$ containing $X$. 
\end{definition}

\begin{remark} In general we have that $\Theta(X) \subseteq \boxminus(X)$. Notice that each magma contains at least a closed submagma, the same magma in itself. Hence, if $M^2$ does not have any propper closed submagma, then $M/\boxminus(X) \cong \I$ for any set $X\subseteq M^2$, since $\boxminus(X)=M^2$.   
In addition, by Theorem~\ref{CaracEquidec}, $M/\boxminus(X)$ is always an equidec magma.

If $M$ is equidec, the diagonal set $\Delta(M)$ is the least closed congruence. Thus, closed congruences on an equidec magma $M$ form a complete algebraic lattice with operations $\theta \cap \theta'$ and $\boxminus(\theta \cup \theta')$, for closed congruences $\theta,\theta'$, and bounds $\Delta(M)$ and $M^2$ (see \cite[p.~21]{sankappanavar1981course} or \cite[p.~48, 147]{davey2002introduction}).  
\end{remark}

\begin{definition} A \emph{presentation} is a pair $(A, R)$ where $R  \subseteq \M_A^2$. Elements in $A$ are called \emph{generators}; elements in $R$, are called \emph{relations}. We say that a presentation is \emph{finite} when the sets $A$ and $R$ are finite. We will write:
\begin{align*}
\langle A \mid R \rangle_\boxminus= \faktor{\M_A}{\boxminus(R)}.
\end{align*}
 We say that an equidec magma $M$ is \emph{finitely presented} iff there is a finite presentation $(A, R)$ such that $M\cong \langle A \mid R\rangle_\boxminus$. 
 
For clarity, we will write the pair $(A,R)$ for a presentation as $(A\mid R)$. As it is usual, to show specific examples, relations $(x,y) \in R$ will be written as $x \approx y$. In addition, since we are going to work only with equidec magmas, we will write $\langle A \mid R \rangle=\langle A \mid R \rangle_\boxminus$.

The free product of two equidec magmas with finite presentations $M=\langle A \mid R \rangle$, and $N=\langle B \mid S \rangle$ is defined as: 
\begin{align*}
M * N=\langle A \sqcup B \mid R \sqcup S \rangle.
\end{align*}
It can be proved that this product does not depend on the presentations. 
\end{definition}

\begin{example} We see some presentations. 
\begin{enumerate}
\item The trivial magma and the free magma have presentations:
\begin{align*}
\I= \langle a \mid a\approx a+a \rangle \,\,\, \mbox{ and } \,\,\, \M_A=\langle A \mid \emptyset \rangle,
\end{align*}
respectively. And as it could be expected:
\begin{align*}
\M_A *\M_B =\M_{A \sqcup B}.
\end{align*}
\item The simplest non-trivial equidec magmas are:
\begin{align*}
\mathcal{C}_{3_+}=\langle a \mid a\approx a+(a+a) \rangle, \,\,\,\, \mathcal{C}_{3_-}=\langle a \mid a\approx (a+a)+a \rangle.
\end{align*}
$\mathcal{C}_{3_+}$ and $\mathcal{C}_{3_-}$ are not isomorphic, but anti-isomorphic. See later Example~\ref{ExampCyclic} for the notation $\mathcal{C}_x$. 

\item Consider the submagma  $\langle \varepsilon, \beta^*\alpha \rangle \subset \Lan$. We have the unique non-trivial relation:
\begin{align*}
\varepsilon \oplus \beta^*\alpha= \alpha \varepsilon \cup \beta\beta^*\alpha=\beta^*\alpha. 
\end{align*}
Thus, this magma can be presented as:
\begin{align*}
\langle a, b \mid b\approx a+b\rangle.
\end{align*}

\item Let us see a free decomposition. Consider the trees $\circ, t \in \Tree$, where $\circ$ is the one-vertex tree, and $t$ is the infinite full binary tree. We have:
\begin{align*}
\langle \circ, t \rangle \cong \langle a,b \mid b\approx b+b \rangle \cong \M * \I.
\end{align*} 
That is, there is no non-trivial equation relating $\circ$ and $t$.
\item We saw in Example~\ref{ExampDiamond} that $\Nat_\diamond$ has only two non-trivial relations 
\begin{align*}
1=1 \diamond 1 \,\,\, \mbox{ and } \,\,\, 2=1\diamond 2.
\end{align*}
Hence,
\begin{align*}
\Nat_\diamond \cong \langle a,b \mid a\approx a+a, \,\,\, b\approx a+b \rangle.
\end{align*}
\end{enumerate}
\end{example}

\begin{notation} We introduce the following notations. 
Let $a \in A, y\in \M_A$. $f_{a\rightarrow y}: \M_A \longrightarrow \M_A$ denotes the homomorphism of free magmas defined over the generators as: 
\begin{align*}
f_{a\rightarrow y}(x)=\begin{cases} x & \mbox{ if } x\not=a,\\
y & \mbox{ if } x=a. \end{cases}
\end{align*}
We will write  $f^2_{a\rightarrow y}((x,y))=(f_{a\rightarrow y}(x),f_{a\rightarrow y}(y))$.  We abbreviate $\g (x)=g_{\M_A^2}(x)$ which, recall, yields the set of the necessary generators in order to generate $x$, see Definition~\ref{DefMinGen}. 
\end{notation}

\begin{definition} \label{DefiReduction} Given a finite presentation $(A,R)$, we say that it is \emph{firstly reducible} iff some of the following transformations is applicable. 
\begin{itemize}

\item[(I)] If there is some $(x,y) \in R$ such that $(x,y) \not \in \G(\M_A^2)$, then we set the presentation:
\begin{align*}
(A' \mid R')=\big(A \mid \g (R )\big).
\end{align*}

\item[(II)] If there are a pair of relations $(a,y), (a,y') \in R \cap (A\times (\M_A \setminus A))$, such that $y\not=y'$, then we set:
\begin{align*}
( A' \mid R' )= \big( A \mid R \setminus \{(a,y')\} \cup \g((y,y')) \big).
\end{align*}

\item[(III)] If there are a pair of relations $(x,b), (x',b) \in R \cap ((\M_A \setminus A) \times A)$, such that $x\not=x'$, then we set:
\begin{align*}
( A' \mid R' )= \big( A \mid R \setminus \{(x',b)\} \cup \g((x,x')) \big).
\end{align*}

\end{itemize}

We say that the presentation is \emph{secondly reducible} iff some of the following transformations is applicable. 

\begin{itemize}
\item[(IV)] If there is some $(x,y) \in R$ such that $x\not \in A$, then we set:
\begin{align*}
(A' \mid R')=( A \mid (R \setminus \{(x,y)\}) \cup \{(y,x)\}).
\end{align*}

\item[(V)] If there is some $a\in A$ such that $(a,y) \not\in R$ and $(a,a) \not \in R$ for all $y \in M_A$, then we set:
\begin{align*}
( A' \mid R' )= ( A \mid R \cup \{(a,a)\}).
\end{align*}

\item[(VI)] If there are a pair of relations $(a,y), (a',y) \in R \cap (A\times \M_A)$ with $a\not=a'$, then we set:
\begin{align*}
( A' \mid R' )= ( A \setminus \{a'\} \mid f_{a'\rightarrow a}^2 \big(R \setminus \{(a',y)\}\big) ).
\end{align*}

\item[(VII)] If there is some $(a,y) \in R \cap (A \times \M_A)$ such that $a \not \in g_{\M_A}(y)$, and $(a,y') \not \in  R$ for any other $y'\not=y$, then we set:
\begin{align*}
( A' \mid R' )=( A \setminus \{a\} \mid f_{a\rightarrow y}^2 \big(R \big)).
\end{align*}
\end{itemize}
When some transformation (I)-(III) is applied we write $( A \mid R)  \leadsto_1 ( A' \mid R' )$ and when some transformation (IV)-(VII) is applied, 
$( A \mid R)  \leadsto_2 ( A' \mid R' )$. And in general, we write $( A \mid R)  \leadsto ( A' \mid R' )$ when some transformation (I)-(VII) is applicable. 
We say that a presentation is \emph{irreducible} iff it is not firstly nor secondly reducible. 
\end{definition}

\begin{lemma} Let  $( A \mid R )$ and $(A'\mid R')$ be finite presentations. 
\begin{align*}
( A \mid R ) \leadsto ( A' \mid R' ) \implies \langle A \mid R \rangle \cong \langle A' \mid R' \rangle.
\end{align*}
\end{lemma}
\begin{proof} We check that transformations (I)-(VII) preserve the magma:
\begin{itemize}
\item[(I)] Let $(x,y) \in R$. A relation $(x,y)$ can be obtained by sums of its generators $\g((x,y))$. Hence, we have that:
\begin{align*}
(x,y) \in \boxminus(R) \iff \g((x,y))\subseteq \boxminus(R).
\end{align*}
For the direction $(\Rightarrow)$ we use that $\boxminus(R)$ is a closed congruence.  For the direction $(\Leftarrow)$ we use that $\boxminus(R)$ is a submagma. 
This means that $\boxminus(R)=\boxminus(\g(R))$, which implies:
$\langle A\mid R \rangle= \langle A \mid \g(R) \rangle = \langle A' \mid R' \rangle$.

\item[(II)] We notice that if $(a,y), (a,y') \in R$ then we can substitute the relation $(a,y')$ by $(y,y')$, since the operator $\boxminus$ will generate $(a,y')$ by the transitivity property. However, as in the case (I), we can substitute the relation $(y,y')$ by the set of relations $\g((y,y'))$, since $\boxminus(R)$ is a closed submagma. Hence, $\boxminus(R)=\boxminus(R')$. 

\item[(III)] this case is symmetric to case (II).

\item[(IV)] Since the operator $\boxminus(R)$ yields an equivalence relation, we can change any relation $(x,y)$ by the relation $(y,x)$ preserving the magma. 

\item[(V)] Since the operator $\boxminus(R)$ yields an equivalence relation, we can add diagonal elements $(a,a)$ preserving the magma. 

\item[(VI)] If we have a pair of relations $(a,y)$ and $(a',y)$ in $R \cap (A\times \M_A)$, then we have a redundant generator, say $a'$. We can remove it, but we have to rewrite all the relations substituting $a'$ by $a$ which is made by $f_{a'\rightarrow a}$. 

\item[(VII)] If $a \not \in g_{\M_A}(y)$, then $a$ does not ``participate'' in $y$. The relation $(a,y)$ just renames $y$ as $a$. It might be the case that there is another relation $(a,y')$, however this case is excluded in the condition of applicability. So we can remove it and rewrite the relations.  \qedhere 
\end{itemize}
\end{proof}

\begin{remark} Transformations (IV)-(VII) are valid for any magma, whereas transformations (I)-(III) work only for the quasivariety of equidec magmas.   
\end{remark}

\begin{lemma} \label{LemaIrred1} Let $(A \mid R)$ a finite presentation. Then, either the presentation is firstly irreducible, or there is a finite sequence of presentations such that:
\begin{align*}
( A \mid R ) \leadsto_1 ( A_1 \mid R_1 )  \leadsto_1 (A_2 \mid R_2) \leadsto_1 \cdots \leadsto_1  (A_n \mid R_n),
\end{align*}
where $( A_n \mid R_n)$ is firstly irreducible. 
\end{lemma}
\begin{proof} Given a set of relations $R$ we define the number: 
\begin{align*}
\alpha(R)=\max\{ \max\{ \ell(x), \ell(y)\} \mid (x,y) \in R \}.
\end{align*}
We prove that if some transformation (I)-(III) is be applied, then the number $\alpha(R)$ decreases. First we notice that if $(x,y) \not \in \G(\M_A^2)$, then:
\begin{align*}
\alpha \big( \g( (x,y)) \big)<\alpha \big(\{(x,y)\} \big).
\end{align*}
 
\begin{itemize}
\item[(I)] By the remark above, we have that $\alpha\big(\g(R)\big)<\alpha(R)$.
\item[(II)] We have the inequalities:
\begin{align*}
\alpha(R') &= \alpha \big( (R \setminus \{(a,y')\}) \cup \g((y,y')) \big)\\
& < \alpha \big( (R \setminus \{(a,y')\}) \cup \{ (y,y')\} \big) \\
& = \alpha(R).
\end{align*}   
In the second line we have used the condition of applicability according which $a \in A$ and $y,y' \not \in A$. Thus, $(y,y') \not \in \G(\M_A^2)$, whereby, by the above comment: $\alpha \big( \g( (y,y')) \big)<\alpha \big(\{(y,y')\} \big)$. Then: 
\begin{align*}
\alpha \big( \, (R \setminus \{(a,y')\} ) \cup \g((y,y')) \, \big) < \alpha \big(  \, (R \setminus \{(a,y')\}) \cup \{ (y,y')\} \, \big).
\end{align*}
In the last line we have used that $\ell(a)=1$, and that $(a,y),(a,y') \in R$. Hence:
\begin{align*}
\alpha \big( \{(y,y')\}\big)=\alpha \big(\{(a,y), (y,y')\} \big)=\alpha \big(\{(a,y), (a,y')\} \big),
\end{align*}
and then:
\begin{align*}
\alpha \big( (R \setminus \{(a,y')\}) \cup \{ (y,y')\} \big) = \alpha(R).
\end{align*}
\item [(III)] Similar to case (II).
\end{itemize}
Since the number $\alpha$ decreases, there are only a finite number of transformations (I)-(III) that can be applied. Therefore, the presentation $P$ can be reduced to a presentation that is firstly irreducible.  
\end{proof}

\begin{lemma} \label{LemaIrred2} Let $(A \mid R)$ a finite presentation firstly irreducible. Then, either the presentation is secondly irreducible, or there is a finite number of presentations such that:
\begin{align*}
( A \mid R ) \leadsto_2 ( A_1 \mid R_1 )  \leadsto_2 ( A_2 \mid R_2 ) \leadsto_2 \cdots \leadsto_2  ( A_n \mid R_n ),
\end{align*}
where $(A_n \mid R_n)$ is firstly and secondly irreducible, that is, it is irreducible. 
\end{lemma}
\begin{proof} Given a presentation $P=( A \mid R )$ we define the numbers:
\begin{align*}
\beta(P)&=\beta_1(P)+\beta_2(P)+\beta_3(P)+\beta_4(P),\\
\beta_1(P)&=|A|,\\
\beta_2(P)&=|R|,\\
\beta_3(P)&=|R\cap ( \M_A \times A)|,\\
\beta_4(P)&=|\Delta(A) \setminus (\Delta(A)\cap R)|.
\end{align*}
$\beta_1$ counts the number of generators; $\beta_2$ counts the number of relations; $\beta_3$ counts the number of relations of the form $(x,a)$ with $a \in A$; $\beta_4$ counts the number of relations that are not in the diagonal.    

We see that if $P \leadsto_2 P'$ for some transformation (IV)-(VII), then the number $\beta$ decreases. We only specify those numbers that change, the rest of numbers remain equal. 
\begin{itemize}
\item[(IV)] $\alpha_3(P')=\alpha_3(P)-1$.  
\item[(V)] $\alpha_4(P')=\alpha_4(P)-1$.  
\item[(VI)] $\alpha_1(P')=\alpha_1(P)-1$, $\alpha_2(P')=\alpha_2(P)-1$. 
\item[(VII)] $\alpha_1(P')=\alpha_1(P)-1$.  
\end{itemize}
Thus, $\beta(P')<\beta(P)$. Since the number $\beta$ always decreases, there are only a finite number of transformations (IV)-(VII) that can be applied. Therefore, the presentation $P$ can be reduced to a presentation secondly irreducible. 
 
Finally we need to see that the presentation is still firstly irreducible. 
Let $P=(A\mid R)$ and $P'=(A'\mid R')$ presentations. We prove that if $P$ is firstly irreducible and $P\leadsto_2 P'$, then $P'$ is still firstly irreducible, or what is the same, no transformation (I)-(III) is applicable.
 
On the one hand, since $P$ is firstly irreducible, transformation (I) is not applicable, which is equivalent to say that $R \subseteq \G(\M_A^2)$. This is because transformations (IV)-(VII) reverse coordinates, add diagonal elements, remove relations or rewrite generators. Therefore, $R' \subseteq \G(\M_A^2)$, whereby (I) is not applicable on $P'$. 

On the other hand, transformation (II) is not applicable iff given $a \in A$, if $(a,y), (a,y') \in R$, then $y=y'$. Transformation (IV) does not affect this condition since it requires that $a \not \in A$. Transformation (V) just introduces a diagonal element, which does not affect the condition. Transformations (VI) and (VII) rewrite a generator, whereby the condition holds. 

Finally we notice that if a presentation is secondly irreducible, the transformation (III) cannot be applied because all the relations are in $A \times \M_A$. 
\end{proof}

\section{$\E$-magmas}
\label{ISoCharac}

\begin{definition} Let $\varphi: A \longrightarrow \M_A$ be a  mapping. We say that $\varphi$ is in \emph{E-form} iff
$\varphi$ is injective and $a \in g_{\M_A}(\varphi(a))$ for each $a\in A$, and then  we write:
\begin{align*}\E(\varphi)= \faktor{\M_A}{ \boxminus(\varphi)}.
\end{align*}
We call these magmas $\E$-\emph{magmas}.  
\end{definition}

\begin{theorem} \label{TheoremCharactFPmagmas} $\E$-magmas are equidec.  Each finitely presented equidec magma is an $\E$-magma. 
\end{theorem}

\begin{example} This theorem synthesizes the last Section~\ref{Presentations}: an irreducible presentation defines a mapping in $E$-form. Let us see an example. Consider the following transformations over the presentation of an equidec magma $M$:
\begin{align*}
 & \big( a,b,c,d \mid (c+d)+(a+c) \approx (a+a)+a  \big)\\
\stackrel{\mbox{\footnotesize (I)}}{ \leadsto}  \,\,\, & \big( a,b,c,d \mid c \approx a, \,\, d \approx a, \,\, a+c \approx a  \big)\\
\stackrel{\mbox{\footnotesize (IV)}}{ \leadsto}  \,\,\, & \big( a,b,c,d \mid c \approx a, \,\, d \approx a, \,\, a \approx a+c  \big)\\
\stackrel{\mbox{\footnotesize (V)}}{ \leadsto}  \,\,\, & \big( a,b,c,d \mid c \approx a, \,\, d \approx a, \,\, a+c \approx a, \,\, b \approx b  \big)\\
\stackrel{\mbox{\footnotesize (VII)}}{ \leadsto}  \,\,\, & \big( a,b,d \mid  d \approx a, \,\, a+a \approx a, \,\, b \approx b  \big)\\
\stackrel{\mbox{\footnotesize (VII)}}{ \leadsto}  \,\,\, & \big( a,b \mid a \approx a+a, \,\, b \approx b  \big).
\end{align*}
The relations in the last presentation form a injective mapping $\varphi: \{a,b\} \longrightarrow \M_{\{a,b\}}$, $\varphi(a)=a +a$, $\varphi(b)=b$.
Hence:
\begin{align*}
M\cong \E(\varphi)=\langle a,b \mid a \approx a+a, \,\, b \approx b\rangle \cong \I *\M. 
\end{align*}
\end{example}

\begin{proof}[Proof of Theorem~\ref{TheoremCharactFPmagmas}.]  The first statement is trivial, since the operator $\boxminus$ yields a closed congruence. 
For the second statement, consider a finitely presented magma $M$. That is, $M=\langle A \mid R\rangle$ for some finite presentation $(A \mid R)$. 
By Lemma~\ref{LemaIrred1} and Lemma~\ref{LemaIrred2} we can assume that the presentation is irreducible, otherwise we reduce the presentation.  

Since the presentation is irreducible, no transformation (I)-(VII) in Definition~\ref{DefiReduction} is applicable. 
Since (I) is not applicable, all the elements of $R$ are in $\G(\M_A^2)$. 
Since (V) is not applicable, $R \subseteq A\times \M_A$. Since (VII) is not applicable, we have that $R$ defines a mapping, maybe partial, $R: A \longrightarrow  \M_A$. Since (VI) is not applicable, the mapping must be total. 
Since (VII) is not applicable, we have that $a \in g_{\M_A}(R(a))$. 
In addition, the unique relations in $A \times A$ are the diagonal elements. 
Since (II) is not applicable, for each $a\in A, y,y' \in \M_A \setminus A$, we have that $(a,y), (a,y') \in R$, implies $y=y'$. Since the unique relations in $A\times A$ are the diagonal elements, the last statement can be extended to any $y,y' \in \M_A$. 
Therefore, the mapping is injective. 
In sum, if $(A \mid R)$ is irreducible, then $R$ is a mapping in E-form. 
\end{proof}

\begin{example} \label{FreeFullEmagmas} Free magmas occur when the mapping $\varphi: A \longrightarrow \M_A$ is the identity $\E(\Delta(A))\cong \M_A$. Actually, the initial part of an $\E$-magma can be calculated as:
\begin{align*}
\Ini \big( \E(\varphi) \big)=\E\big(\varphi \cap \Delta( A) \big).
\end{align*}
In particular, by Theorem~\ref{TheoSplit}, $\E(\varphi)$ is full when $\varphi \cap \Delta(A)=\emptyset$ and it is free when $\varphi \cap \Delta(A)=\varphi$.
\end{example}

\begin{example} \label{PropositionProducts} The free product of two $\E$-magmas is easy to calculate. Let $\varphi: A \longrightarrow$ and $\psi: B \longrightarrow \M_B$ be mappings in E-form. The set $\varphi \sqcup \psi$ defines a mapping in E-form:
\begin{align*}
\big(\varphi \sqcup \psi \big) (x)=\begin{cases} \varphi(x) \mbox{ if } x\in A,\\ 
\psi(x) \mbox{ if } x\in B. \end{cases}
\end{align*} 
Then, we have:
\begin{align*}
\E(\varphi) * \E(\psi) \cong \E(\varphi \sqcup \psi),
\end{align*}
which states that the free product of $\E$-magmas is an $\E$-magma.
\end{example}

\begin{remark} $\E$-magmas form a large class of equidec magmas. 
Although the direct product of two equidec magmas is equidec, we cannot ensure here that the direct product of two $\E$-magmas is an $\E$-magma. That is, we do not know if there is a mapping in E-form, say $\phi$, for which $\E(\varphi)\times \E(\psi) \cong \E(\phi)$, beyond some trivial cases. The difficulty reflects the fact that the direct product is not in general finitely generated, recall Theorem~\ref{TeoProduct}.
\end{remark}

\begin{example} \label{ExempNonEmagma} Let us consider an example of equidec magma but possibly non $\E$-magma. Let the alphabet be $\Sigma=\{p,q\}$. The following magma:
$$\langle \Sigma^* \mid x\approx px+qx, \,\, x \in \Sigma^* \rangle$$
is a non-finitely represented equidec full magma. The set of relations defines a injective mapping $\Sigma^* \longrightarrow \M_{\Sigma^*}$, but it is not in $E$-form since $x \not \in g_{M_{\Sigma^*}}(\varphi(x))$. 
This presentation has an interesting property: every generator is superfluous. We can remove any generator $x$ since it can be generated by $px$ and $qx$. However, we cannot remove an infinite set of generators. 

When we consider its J\'onsson-Tarski algebra (recall Example~\ref{JonssonTarskiAlg}), the operations $p,q$ for the first and second component are $p(x)=px$ and $q(x)=qx$, for each element $x$ of the magma. This structure is isomorphic to the cyclic free J\'onsson-Tarski algebra which is generated by the empty word $\varepsilon$. 
\end{example}

\begin{lemma} \label{LemaRank} Consider the equidec magma $\E(\varphi)$ where $\varphi: A \longrightarrow \M_A$.The projection $\pi:A \longrightarrow \pi(A)$ is a bijection. 
\end{lemma}
\begin{proof} First we notice that given a subset $X \subseteq \M_A^2$, any element in $\boxminus(X) \setminus \Xi(X)$ is always decomposable. If $x \in \boxminus(X)$ is indecomposable, then the unique ``part'' of the operator $\boxminus$ that could contribute to generate $x$ is $\Xi(X) \subseteq \boxminus (X)$.  

Let us suppose that $\pi(a)=\pi(b)$ for some $a,b \in A$, that is, $(a,b) \in \Theta(\varphi)$. $(a,b)$ is clearly indecomposable, therefore, by the above comment,
$(a,b) \in \Xi(X)$. We have that for any $x\in A$:
\begin{align*}
\Xi \big( (x,\varphi(x) \big)=\{(x,x),(\varphi(x),\varphi(x)), (x,\varphi(x)), (\varphi(x),x)\}.
\end{align*}
In addition, given $x\in A$, we have that $\varphi(x) \in A$ iff $\varphi(x)=x$, since $x\in g_{\M_A}(\varphi(x))$. Then, 
we can calculate $\Xi(\varphi)$ as:
\begin{align*}
\Xi(\varphi)=\Delta(A) \cup \Delta(\varphi(A)) \cup \varphi \cup \overline{\varphi},
\end{align*}
where $\overline{\varphi}=\{ (y,x) \mid (x,y) \in \varphi\}$. And now it is easy to check that if $a\not= b$, then $(a,b)$ is not in any of those subsets. Thus, $\pi:A \longrightarrow \pi(A)$ is injective, and therefore bijective.  
\end{proof}

\begin{definition} Let $\varphi, \psi$ be mappings $\varphi: A \longrightarrow \M_A$ and $\varphi: B \longrightarrow \M_B$. We say that $\varphi$ and $\psi$ are \emph{conjugated} iff there is an isomorphism of free magmas $g: \M_A \longrightarrow \M_B$ that makes commutative the diagram:
\begin{align*}
\xymatrix{A \ar[d]_\varphi \ar[r]^g & B \ar[d]^\psi  \\
\M_A  \ar[r]_g & \M_B}
\end{align*} 
\end{definition}

\begin{theorem} $\E(\varphi)\cong \E(\psi)$ iff $\varphi, \psi$ are conjugated.
\end{theorem}
\begin{proof} 
$(\Rightarrow)$ Let $f: \E(\varphi) \longrightarrow \E(\psi)$ be the isomorphism. First we prove that there is an isomorphism of free magmas $g$ such that the diagram commutes:
\begin{align*}
 \xymatrix{\M_A \ar[d]_{\pi_A} \ar[r]^g & \M_B \ar[d]^{\pi_B}  \\
\E(\varphi) \ar[r]_f & \E(\psi)}
\end{align*} 
By Lemma~\ref{LemaRank}, $\pi_A: A \longrightarrow \pi(A)$ and $\pi_B: B\longrightarrow \pi(A)$ are bijective.  
This allows us to define a mapping $g:A \longrightarrow B$ as: 
\begin{align*}
g(a)= (\pi^{-1}_B \circ f \circ \pi_A)(a).
\end{align*} 
$g$ so defined over the generators defines a homomorphism $g:\M_A \longrightarrow \M_B$. $g$ clearly commutes in the diagram since we have for the generators that:
\begin{align*}
g(a)=b \iff f([a]_A)=[b]_B \iff f(\pi_A(a))=\pi_B(b)=\pi_B(g(a)), 
\end{align*} 
and, since $f,g,\pi_A,\pi_B$ are homomorphisms, this can be extended to any element $x\in \M_A$. Hence $f\circ \pi_A=\pi_B \circ$.
Since $g:A \longrightarrow B$ is bijective, $g:\M_A \longrightarrow \M_B$ is an isomorphism.  
Now we need to see that $\varphi$ and $\psi$ are conjugated. Since we deal with homomorphisms, we only need to see for the generators. We have:
\begin{align*}
&f([a]_A)=[b]_B \\
\implies &f([\varphi (a)]_A)=[\psi(b)]_B \\
\implies &g(\varphi (a))=\psi(b) \\
\implies &g(\varphi (a))=\psi(g(a)). 
\end{align*}

$(\Leftarrow)$ Let $g$ be the isomorphism of free magmas $g: \M_A \longrightarrow \M_B$, such that $g \circ \varphi=\psi \circ g$. We write $\big( g\times g \big)(x,y)=\big( g(x),g(y) \big)$. 
It is clear that if $\theta$ is a congruence on $\M_A$, then:
\begin{align*}
\faktor{\M_A}{\theta}\cong \faktor{g(\M_A)}{(g\times g)(\theta)}.
\end{align*}
Now we see that, considering $\varphi=\{ \big(a, \varphi(a)\big) \mid a\in A\}$ as a subset of $\M_A^2$, we have:
\begin{align*}
(g \times g)(\varphi) &=\{ \big(g(a), (g(\varphi(a)) \big) \mid a\in A\}\\
&=\{\big( g(a), \psi(g(a)) \big) \mid a\in A\}\\
&=\{(b, \psi(b)) \mid b\in B\}\\
&=\psi.
\end{align*}
Therefore, $(g \times g)\big(\boxminus(\varphi)\big)=\boxminus \big( (g\times g)(\varphi) \big)=\boxminus(\psi)$. And then:
\begin{align*}
\E(\varphi)&=\faktor{\M_A}{\boxminus(\varphi)} \cong \faktor{g(\M_A)}{(g\times g)\big(\boxminus(\varphi)\big)}=\faktor{\M_B}{\boxminus(\psi)}=\E(\psi). \qedhere
\end{align*}
\end{proof}

\begin{example} \label{ExampCyclic} Cyclic equidec magmas have at most one non-trivial relation, which allows to write them $\mathcal{C}_x= \E(1 \approx x)$, where $x\in \M$. The first cyclic equidec magmas are:
\begin{align*}
\mathcal{C}_1\cong \M, \,\,\, \mathcal{C}_{2}\cong \mathcal{I},\,\,\, \mathcal{C}_{3_-}\cong \langle 1\mid 1\approx 3_- \rangle,\,\,\, \mathcal{C}_{3_+}\cong \langle 1\mid 1\approx 3_+ \rangle, \ldots
\end{align*} 
where we recall Notation~\ref{NotationCardo} for the elements in $\M$. 
We have that: 	 
\begin{align*}
\mathcal{C}_x \cong \mathcal{C}_y \iff x=y.
\end{align*} 
Thus, for each $x\in \M$ we have a unique cyclic equidec magma.   
\end{example}

\begin{remark} \label{LastJonssonTarski}
Let $\mathbf{K}$ be a class of algebras. It is said that an algebra $A$ is $\mathbf{K}$-\emph{freely generated by the set} $X$ iff $X$ generates $A$ and if for every algebra $B \in \mathbf{K}$, every mapping $f:X\longrightarrow B$ can be extended to a homomorphism of algebras of $A$ into $B$.  
Let $\mathbf{K}_2$ be the class of J\'onsson-Tarski algebras. In \cite{Jonsson1961Properties} (Theorem~5) it is proved that if two J\'onsson-Tarski algebras are finitely $\mathbf{K}_2$-freely generated then they must be isomorphic. 

Notice that, in particular, the isomorphism of these algebras is preserved into magmas, that is, if $\widetilde{M} \cong \widetilde{N}$, then $M \cong N$. Notwithstanding, the examples of J\'onsson-Tarski algebras from a full and equidec magma in this article are not $\mathbf{K}_2$-freely generated, with the only exception of the Example~\ref{ExempNonEmagma}. 

Let us see, as illustration, that the J\'onsson-Tarski algebra $\widetilde{\Nat}_\diamond$  of Example~\ref{JonssonTarskiAlg} is not $\mathbf{K}_2$-freely generated. Every cyclic equidec magma is full with the only exception of $\mathcal{C}_1\cong \M$, recall Example~\ref{FreeFullEmagmas}. Consider the particular case of the magma $\mathcal{C}_{3_+}$. In this magma the equation $[1]=x+[1]$ has no solution, otherwise we would have two different relations in the magma $[1]=[1]+([1]+[1])$ and $[1]=x+[1]$, which makes no sense.
We saw that $\widetilde{\Nat}_\diamond$ is generated by the integer $2$ which satisfied the relation $2=1\diamond 2$. Now we consider a mapping $f: \{2\} \longrightarrow \widetilde{\mathcal{C}}_{3_+}$ defined as $f(2)=[1]$. We see that $f$ cannot be extended to a homomorphism $f: \widetilde{\Nat}_\diamond \longrightarrow \widetilde{\mathcal{C}}_{3_+}$, otherwise we would have $[1]=f(2)=f(1\diamond 2)=f(1)+f(2)=f(1)+[1]$, for some value $f([1]) \in \mathcal{C}_{3_+}$, and this is absurd. 

In the proof of Theorem~5 in \cite{Jonsson1961Properties}, J\'onsson and Tarski used a general form of the algebra $\widetilde{\Nat}_\diamond$ to prove that the class $\mathbf{K}_2$ contains non-trivial algebras, although the $\mathbf{K}_2$-freeness of that algebra is not required for the proof. 
By similar means one can prove that cyclic magmas and other full and equidec magmas seen in this article are not $\mathbf{K}_2$-freely generated. 
\end{remark}

\section*{Acknowledgment}
Many thanks to the anonymous reviewer for so exhaustive and accurate revision, specially for shortening the proof in the Example~\ref{ExampLan} and the reference of the article of J\'onsson and Tarski.


\end{document}